\documentclass[11pt]{article}
\usepackage[T1]{fontenc}
\usepackage{lmodern,amsmath,amssymb,amsthm,mathtools,microtype}
\usepackage[margin=1in]{geometry}
\usepackage[colorlinks=true,linkcolor=blue,citecolor=blue,urlcolor=blue]{hyperref}
\usepackage{booktabs}
\usepackage{enumitem}
\allowdisplaybreaks[2]
\numberwithin{equation}{section}
\newtheorem{theorem}{Theorem}[section]
\newtheorem{lemma}[theorem]{Lemma}
\newtheorem{proposition}[theorem]{Proposition}
\newtheorem{corollary}[theorem]{Corollary}
\theoremstyle{definition}\newtheorem{definition}[theorem]{Definition}
\theoremstyle{remark}\newtheorem{remark}[theorem]{Remark}
\newcommand{\E}{\mathbb E}
\newcommand{\Prb}{\mathbb P}
\newcommand{\Ent}{\mathcal E}
\newcommand{\TV}{\mathrm{TV}}
\newcommand{\mix}{\mathrm{mix}}
\newcommand{\ind}{\mathbf 1}
\newcommand{\id}{\mathrm{id}}
\newcommand{\sgn}{\mathrm{sgn}}
\newcommand{\HS}{\mathrm{HS}}
\newcommand{\norm}[1]{\left\lVert#1\right\rVert}

\newcommand{\Ftwo}{\mathbb F_2}

\hypersetup{pdftitle={A subquadratic upper bound for the mixing time of the Thorp shuffle},pdfauthor={Yunjiang Jiang}}
\title{A subquadratic upper bound for the mixing time of the Thorp shuffle}
\author{Yunjiang Jiang}
\date{October 3, 2026}
\begin{document}
\maketitle

\begin{abstract}
Let $n=2^d$.  We prove that the worst-start total-variation mixing time of the
full labelled Thorp shuffle satisfies, for every fixed $0<\varepsilon<1$,
\[
 t_{\mathrm{mix}}(\varepsilon)
 \le C_* d^{5/3}+O_\varepsilon\!\left(d^{4/3}\log(d+2)\right),
 \qquad C_*=5.1536951454\ldots .
\]
Time is measured in ordinary Thorp shuffles.  This improves the previously
published full-deck upper bound $O(d^3)$ of Morris.  The geometric starting
point is elementary: label the $n=2^d$ positions by binary strings.  One
physical Thorp shuffle is a random swap layer on one coordinate matching of
the $d$-dimensional cube, followed by a deterministic cyclic rotation of the
coordinates.  In a rotating frame that deterministic rotation disappears, so
$d$ successive physical shuffles become one \emph{sweep}: one independent
random swap layer in each coordinate direction.  The proof then combines five
ingredients.  First, a collision-graph representation and negative dependence
give pointwise smoothing for ordered marked cards.  Second, an entropy argument
is sharpened by moving an endpoint moment up to the critical exponent
$\beta_c=\log_2(3/2)$, yielding an entropy-preparation clock
$s^2/(2\log(3/2))+O(s\log s)$.  Third, exact multiplicities in the ordered
$h$-tuple permutation representation amplify marked-card bounds into Fourier
operator and energy estimates on low representation levels.  Fourth, entropy
truncation produces a positive mixture whose regular component retains the
marked-card budget.  Finally, only two long blocks need this entropy
regularization: a shorter third block is used solely for additional low-level
smoothing.  Optimizing the three block lengths gives the constant above.  We
state all conventions and reproduce the shuffle-specific analytic estimates in
full; the only black-box inputs are standard strong-Rayleigh negative
dependence and a standard high-level dimension estimate for representations of
$S_n$.
\end{abstract}



\tableofcontents

\section{Introduction and historical context}\label{sec:introduction}

The Thorp shuffle is a particularly local model of a riffle shuffle.  For an
even number $n$ of cards, split the deck into two equal packets.  For each pair
of current bottom cards, toss a fair coin to decide whether the left card or the
right card is dropped first, and then drop the other card.  Continue until both
packets are empty.  Edward Thorp introduced this model in 1973 in his study of
nonrandom shuffling and the game of Faro \cite{Thorp1973}.  In contrast with
the Gilbert--Shannon--Reeds riffle shuffle, whose total-variation behavior was
determined sharply by Bayer and Diaconis \cite{BayerDiaconis1992}, the Thorp
shuffle resisted analysis for decades.

For $n=2^d$, the first polylogarithmic full-deck upper bound was obtained by
Morris: $O(d^{44})$ \cite{Morris2008}.  Montenegro and Tetali sharpened this to
$O(d^{29})$ using spectral-profile/evolving-set methods
\cite{MontenegroTetali2006}.  Morris subsequently proved an $O((\log n)^4)$
bound for arbitrary even $n$ \cite{Morris2009AOP}, and then, for $n=2^d$,
reduced the full-deck bound to $O(d^3)$ \cite{Morris2013}.  In a different
direction, Morris, Rogaway and Stegers obtained strong estimates for projected
Thorp shuffles, showing that very large proper subsets of the cards mix much
faster than the whole deck \cite{MorrisRogawayStegers2009}.  Czumaj and V\"ocking
later proved an $O((\log n)^2)$ statement for linear-size partial permutations;
that theorem does not by itself give an $O((\log n)^2)$ full-labelled mixing
bound \cite{CzumajVoecking2014}.  Recent work on permutation sampling still
cites the cubic Thorp bound as the full-deck benchmark
\cite{AlekseevEtAl2026}.  Thus, to the author's knowledge, the result below is
the first full-deck improvement in the exponent since the cubic bound.

For convenience, the history for powers of two is summarized in
Table~\ref{tab:history}.  The entries refer to full-deck mixing unless noted.
\begin{table}[ht]
\centering
\begin{tabular}{lll}
\toprule
Work & setting & upper bound / conclusion \\
\midrule
Morris \cite{Morris2008} & $n=2^d$ & $O(d^{44})$ \\
Montenegro--Tetali \cite{MontenegroTetali2006} & $n=2^d$ & $O(d^{29})$ \\
Morris \cite{Morris2009AOP} & even $n$ & $O((\log n)^4)$ \\
Morris \cite{Morris2013} & $n=2^d$ & $O(d^3)$ \\
Morris--Rogaway--Stegers \cite{MorrisRogawayStegers2009} & projections & rapid partial mixing \\
Czumaj--V\"ocking \cite{CzumajVoecking2014} & linear-size partial permutations & $O((\log n)^2)$ \\
This paper & $n=2^d$ & $C_*d^{5/3}+O_\varepsilon(d^{4/3}\log d)$ \\
\bottomrule
\end{tabular}
\caption{Selected milestones for the Thorp shuffle.}
\label{tab:history}
\end{table}

There is a simple information-theoretic lower bound of order $d$.  One ordinary
Thorp shuffle uses exactly $n/2$ independent fair bits, so after $r$ shuffles a
fixed initial deck can have support of size at most $2^{rn/2}$.  If its law is
within total variation $\varepsilon<1$ of uniform on $S_n$, its support must
have size at least $(1-\varepsilon)n!$.  Hence
\begin{equation}\label{eq:support-lower}
 r\ge \frac{2}{n}\log_2((1-\varepsilon)n!)
   =2d-2\log_2 e+o(1).
\end{equation}
Thus a substantial gap remains between the linear lower bound and the
$d^{5/3}$ upper bound proved here.

\subsection{Main theorem}
Set
\[
 L=\log 2,\qquad A=\frac1{2\log(3/2)},\qquad
 z_* = \frac{17+\sqrt{97}}8,
\]
and
\begin{equation}\label{eq:Cstar-intro}
 C_*=
 \left(2+\frac1{z_*-2}\right)A^{1/3}(z_*L)^{2/3}
 =5.153695145414931\ldots .
\end{equation}
All logarithms are natural unless a base is displayed explicitly.

\begin{theorem}[Main theorem]\label{thm:main-intro}
Let $n=2^d$, and let $t_{\mix}(\varepsilon)$ denote the worst-start
$\varepsilon$-mixing time in total variation, measured in ordinary Thorp
shuffles.  For every fixed $0<\varepsilon<1$,
\begin{equation}\label{eq:main-intro}
 t_{\mix}(\varepsilon)
 \le C_*d^{5/3}+O_\varepsilon\!\left(d^{4/3}\log(d+2)\right).
\end{equation}
\end{theorem}

The proof is quantitative throughout.  The value $C_*$ is not asserted to be
optimal for the shuffle.  It is the exact minimizer within the asymmetric
three-block schedule used in Section~\ref{sec:three-block}.

\subsection{From the physical shuffle to the hypercube, step by step}
\label{subsec:clock}
This subsection spells out the basic geometric model in detail.  Nothing from
representation theory or entropy is needed here.

\paragraph{The physical procedure.}
Assume $n=2^d$.  Split the deck into two packets of $n/2$ cards, keeping the
order inside each packet.  Pair the cards having the same depth in the two
packets.  For every pair toss an independent fair coin, and use that coin to
decide which member of the pair occupies the earlier of the two corresponding
output positions.  This is one Thorp shuffle.  Thus there are $n/2$ independent
coin tosses in one shuffle.

The binary model records exactly the same operation.  The binary strings label
\emph{positions}, not card identities.  Label the $n$ positions by the vertices
of the cube
\[
 Q_d=\Ftwo^d=\{0,1\}^d.
\]
Write a vertex as $(b,x)$, where $b\in\Ftwo$ is the first bit and
$x\in\Ftwo^{d-1}$ contains the remaining $d-1$ bits.  The two packets are
$b=0$ and $b=1$.  Positions $(0,x)$ and $(1,x)$ form one pair.
For each $x$, let $\xi_x\in\{0,1\}$ be the fair coin assigned to that pair.
Before any deterministic relabelling, the random part of the shuffle is
\begin{equation}\label{eq:first-layer}
 L_\xi(b,x)=(b\oplus\xi_x,x),
\end{equation}
where $\oplus$ is addition modulo two.  If $\xi_x=0$ the two positions keep
their order; if $\xi_x=1$ they are exchanged.  Hence $L_\xi$ is simply a
collection of independent random swaps on all cube edges parallel to the first
coordinate.

The physical output order moves the packet bit from the front of the binary
address to the back.  Define the deterministic cyclic rotation
\begin{equation}\label{eq:rotation}
 R(b_1,b_2,\ldots,b_d)=(b_2,\ldots,b_d,b_1).
\end{equation}
One physical Thorp shuffle is therefore the position permutation
\begin{equation}\label{eq:physical-factorization}
 T_\xi=R\circ L_\xi,
 \qquad
 (b,x)\longmapsto (x,b\oplus\xi_x).
\end{equation}
Equation \eqref{eq:physical-factorization} is not a different shuffle; it is
just the usual two-packet procedure written in binary coordinates.

\paragraph{A concrete eight-card example.}
Take $d=3$.  The eight positions are
\[
 000,\ 001,\ 010,\ 011,\ 100,\ 101,\ 110,\ 111.
\]
The first random layer pairs
\[
 000\leftrightarrow100,\quad
 001\leftrightarrow101,\quad
 010\leftrightarrow110,\quad
 011\leftrightarrow111.
\]
Consider the card currently at $101=(1,01)$.  Its pair coin is $\xi_{01}$.
If $\xi_{01}=0$, the random layer leaves the card at $101$ and the rotation
sends it to $011$.  If $\xi_{01}=1$, the random layer moves it to $001$ and
the rotation sends it to $010$.  Thus the compact formula
$(b,x)\mapsto(x,b\oplus\xi_x)$ literally describes the physical operation.

\paragraph{Removing the deterministic rotation.}
The rotation $R$ does not create or destroy randomness.  It only renames the
positions, so applying a deterministic power of $R$ to every configuration
preserves total-variation distance from the uniform distribution.  We may
therefore observe the chain in a \emph{rotating frame}.

Let $L^{(i)}$ denote a fresh random swap layer on all edges parallel to
coordinate $i$ of the cube; the swaps on different edges are independent.
If the first physical shuffle uses $L_{\xi_1}$, the second $L_{\xi_2}$, and so
on, then after $r$ physical shuffles
\begin{align}\label{eq:rotating-frame-product}
 R^{-r}T_{\xi_r}\cdots T_{\xi_1}
 &=(R^{-(r-1)}L_{\xi_r}R^{r-1})\cdots
   (R^{-1}L_{\xi_2}R)L_{\xi_1}.
\end{align}
Conjugating by $R$ merely changes which coordinate is called ``first.''
Consequently the factors on the right of \eqref{eq:rotating-frame-product}
are independent random swap layers whose coordinate directions cycle through
$1,2,\ldots,d,1,2,\ldots$.

\begin{definition}[Coordinate layer and sweep]\label{def:sweep}
A \emph{coordinate-$i$ layer} independently swaps, with probability $1/2$,
the two endpoints of every cube edge parallel to coordinate $i$.  A
\emph{sweep} is one coordinate layer in each of the $d$ directions, in a fixed
cyclic order.
\end{definition}

Since $R^d$ is the identity, a block of exactly $d$ physical Thorp shuffles is
not merely comparable to a sweep: after the rotating-frame cancellation it is
exactly one sweep, with independent fresh switch bits in all $d$ coordinate
layers.  Thus
\begin{equation}\label{eq:clock-conversion}
 \boxed{\text{one sweep}=d\text{ ordinary Thorp shuffles}.}
\end{equation}
All analytic estimates below are stated in sweeps.  The last step of the proof
multiplies the sweep count by $d$ to return to the physical clock.

\paragraph{Why this viewpoint is useful.}
In the original two-packet description it is difficult to see how several
labelled cards can interact over many shuffles.  In the cube, each layer is a
perfect matching and a card can only cross one edge of the currently active
coordinate.  Two marked cards interact only when they arrive at the two ends
of the same active edge.  This makes their interactions countable by a
\emph{collision graph}.  It also makes the one-card behavior transparent:
during a sweep each coordinate of a given card is refreshed by a fair choice
exactly once.  Conditional on the card's entire path up to that layer, the
switch bit on the edge it now occupies is still a fresh independent fair bit.
Inducting over the $d$ coordinate layers shows that the final $d$ bits are
independent fair bits.  Thus one marked card is exactly uniform on $Q_d$ after
one sweep.  The difficulty of full-deck mixing is therefore not making one
card random; it is destroying the residual dependence among many cards.

For the eight-card example, the rotating-frame directions used by the first
three physical shuffles are coordinate $1$, then coordinate $2$, then
coordinate $3$.  Those three layers form one sweep.  The fourth physical
shuffle starts the next sweep again at coordinate $1$.

\section{Notation and elementary facts}\label{sec:notation}

A deck arrangement is a permutation of the $n$ labels among the $n$ cube
vertices.  We identify the state space with the symmetric group $S_n$; our
composition convention is the one encoded by the convolution formula below.
The particular convention is immaterial for mixing because inversion and
left/right translation preserve the uniform distribution.

Let $U$ denote uniform probability on $S_n$.  If $\nu$ is a probability law on
$S_n$, its density with respect to $U$ is the function
\[
 f(\sigma)=\frac{\nu(\sigma)}{U(\sigma)}=n!\,\nu(\sigma).
\]
Thus uniformity is represented by the constant function $f\equiv1$, and
$\E_U f=1$.  The advantage of this normalization is that distances from
uniformity become ordinary norms of $f-1$.

For densities $f,g$ we use normalized convolution
\[
 (f*g)(x)=\E_{y\sim U} f(xy^{-1})g(y),
\]
so that convolution of densities corresponds to composing independent random
permutations.  If $q_r$ is the density of the walk after $r$ complete sweeps
started from the identity, then $q_r*q_s=q_{r+s}$.  Thus an expression such as
$q_u*q_t*q_t$ is the law of three genuine consecutive blocks whose total
length is $2t+u$ sweeps.

For a density $f$ define
\[
 \Ent(f)=\E_U[f\log f],\qquad
 \norm f_p=(\E_U|f|^p)^{1/p}.
\]
For a random variable taking values in a finite set equipped with a specified
uniform law, $\Ent(X)$ means the relative entropy of its distribution with
respect to that uniform law.  For distributions $p,q$ on the same finite set,
$D(p\Vert q)=\sum_xp(x)\log(p(x)/q(x))$ with the usual conventions.
We write $I(X;Y\mid Z)$ for conditional mutual information.

For two probability measures $\nu,\pi$ on a finite set,
\[
 \norm{\nu-\pi}_{\TV}=\frac12\sum_x|\nu(x)-\pi(x)|.
\]
If $f$ is the density of $\nu$ relative to $\pi$, then Cauchy--Schwarz gives
\begin{equation}\label{eq:TV-L2}
 \norm{\nu-\pi}_{\TV}\le\frac12\norm{f-1}_2.
\end{equation}
If $X_r^{\rm phys}$ denotes the deck after $r$ ordinary Thorp shuffles, define
\[
 t_{\mix}(\varepsilon)=
 \min\left\{r:\max_{\sigma\in S_n}
 \norm{\mathcal L_\sigma(X_r^{\rm phys})-U}_{\TV}\le\varepsilon\right\}.
\]
Here $\mathcal L_\sigma$ means that the chain starts from arrangement $\sigma$.
Left or right translation on $S_n$ preserves both $U$ and total variation, so
it is enough to analyze the walk started from the identity.  We also write
$\log^+x=\max\{\log x,0\}$.

For ordered distinct $h$-tuples of vertices let
\[
 \Omega_h=\{(x_1,\ldots,x_h)\in(\Ftwo^d)^h:x_i\ne x_j\ (i\ne j)\},
 \qquad N_h=(n)_h.
\]
When a permutation law has density $f$, $P_h^f$ denotes its stochastic
transition matrix on $\Omega_h$.  For one sweep we abbreviate this matrix to
$P_h$.  Hilbert--Schmidt norms of finite matrices are the unnormalized sums
\[
 \norm M_{\HS}^2=\sum_{x,y}|M(x,y)|^2.
\]
This convention is important in Section~\ref{sec:tuple-amplification}.

The proof also uses the standard representation theory of $S_n$.  A
\emph{partition} $\lambda\vdash n$ is a nonincreasing list
$\lambda_1\ge\lambda_2\ge\cdots$ of nonnegative integers summing to $n$.
Its Young diagram has $\lambda_i$ boxes in row $i$; the conjugate partition
$\lambda'$ is obtained by reflecting that diagram, so $\lambda'_1$ is the
first column length.  Irreducible representations of $S_n$ are indexed by
partitions.  We denote the representation by $V_\lambda$ and its dimension by
$f^\lambda=\dim V_\lambda$.

For this paper the important statistic is the \emph{level}
\begin{equation}\label{eq:level-intro}
 \ell(\lambda)=\min\{n-\lambda_1,n-\lambda'_1\}.
\end{equation}
A small level means that the Young diagram is close either to one long row or
to one long column.  Informally, these are the Fourier modes that can already
be detected by following only a small number of labelled cards.  Level zero
consists of the constant mode and the sign mode.  This intuition is made exact
in Section~\ref{sec:tuple-amplification}.

\subsection{A short dictionary for first-time readers}
The proof moves among several languages.  The following dictionary may be
useful on a first reading.

\begin{center}
\begin{tabular}{p{0.23\textwidth}p{0.69\textwidth}}
\toprule
Object & Meaning in this proof \\
\midrule
$\Ent(f)$ & How much information about the starting arrangement remains beyond a uniform deck.  It is zero exactly at uniformity. \\
$h$ marked cards & A test that ignores all but $h$ specified labels.  Small-$h$ tests are much easier to analyze geometrically. \\
$P_h$ & The transition matrix seen by those $h$ ordered labels during one sweep. \\
Representation level $k$ & A Fourier scale on $S_n$.  Low levels are controlled by ordered tuples; high levels are controlled using dimension growth. \\
Truncation $F=pf+(1-p)b$ & An exact decomposition of a mildly nonuniform block into a large ``regular'' part $f$ and a rare exceptional part $b$. \\
Long/short blocks & Actual consecutive pieces of the shuffle.  Two long blocks pay for entropy regularization; the third only supplies extra low-level contraction. \\
\bottomrule
\end{tabular}
\end{center}

No step treats these as metaphors: each row is converted below into a precise
inequality.  The table only indicates why that inequality is being proved.

\subsection{Two external inputs}\label{subsec:external}
Everything specific to the Thorp shuffle is proved below.  We use two standard
external facts.

\begin{enumerate}[label=(E\arabic*)]
\item \textbf{Strong Rayleigh negative dependence.}
The only consequence ultimately needed is intuitive: after independently
randomizing by transpositions, occupation of one vertex cannot make occupation
of a disjoint collection of vertices more likely.  In particular, if all
one-vertex occupation probabilities are at most $\theta$, then negative
association gives
\[
 \Prb(B\text{ is entirely occupied})\le\theta^{|B|}.
\]
For completeness, the standard formal language yielding this fact is as
follows.  For a probability measure $\mu$ on subsets of a finite ground set
$V$, its multiaffine generating polynomial is
\[
 Z_\mu(z)=\sum_{A\subseteq V}\mu(A)\prod_{v\in A}z_v.
\]
A real multivariate polynomial is called \emph{stable} if it is nonzero
whenever every variable has strictly positive imaginary part; $\mu$ is
\emph{strongly Rayleigh} when $Z_\mu$ is stable.  If $\tau$ is a
transposition of two ground-set elements, the operation
$\mu\mapsto(1-\theta)\mu+\theta\,\tau\mu$, $0\le\theta\le1$, is
called partial symmetrization.  Partial symmetrization preserves the strongly
Rayleigh property, and strongly Rayleigh measures are negatively associated:
for increasing functions depending on disjoint sets of occupation indicators,
their covariance is nonpositive.  These are Theorems 4.20 and 4.9 of
Borcea--Br\"and\'en--Liggett \cite{BorceaBrandenLiggett2009}.

\item \textbf{High representation levels.}
This input says that once $f$ has controlled $L^2$ norm, Fourier modes of very
high level are automatically tiny because the corresponding representations
have large dimension.  Formally, for a probability density $f$ on $S_n$ and
$1\le K<n/10$, if $T_f$ is
convolution by $f$ and $T_f|_{>K}$ denotes its restriction to the sum of
isotypic components of level $>K$, then
\begin{equation}\label{eq:external-high-level}
 \norm{T_f}_{>K}\le \norm f_2\left(\frac{eK}{n}\right)^{K/2}.
\end{equation}
This follows from the dimension estimate used in the proof of
Keevash--Lifshitz \cite[Theorem 4.4 and Lemma 4.5]{KeevashLifshitz2023}; only
the dimension part is used here, not their globalness hypothesis.
\end{enumerate}

\section{Overview of the proof}\label{sec:overview}

\subsection*{The central difficulty}
One card becomes perfectly uniform after a single sweep, but the whole deck is
far from uniform because the cards share switches and therefore remain highly
dependent.  The proof has to pass from ``a few cards look random'' to ``the
entire permutation looks random.''  The five modules below perform that
passage.

It helps to separate two tasks.  The \emph{geometric task} is to understand a
small set of marked cards moving through the cube.  The \emph{global task} is
to convert those small-set estimates into a statement about all $n!$ deck
arrangements.  Collision graphs solve the first task.  Entropy, truncation,
and representation theory solve the second.

\begin{center}
\begin{tabular}{p{0.24\textwidth}p{0.31\textwidth}p{0.35\textwidth}}
\toprule
Question & Tool & Output used later \\
\midrule
How can $h$ specified cards fail to separate? & Collision graph on the cube & Pointwise bound on the $h$-card transition kernel \\
How quickly does a whole deck forget information? & One-sweep entropy loss & A time after which the deck density is not too spiky on average \\
How do $h$-card estimates control full-deck Fourier modes? & Ordered-tuple representations & Low-level operator and energy bounds \\
What about rare very large density values and high Fourier levels? & Entropy truncation plus dimension growth & A large regular component with controlled $L^2$ norm \\
How should the actual shuffle time be allocated? & Two long blocks plus one shorter block & The optimized $d^{5/3}$ bound \\
\bottomrule
\end{tabular}
\end{center}

\paragraph{1. Marked-card smoothing from collision graphs.}
Fix $h$ labelled cards and temporarily ignore all the others.  If their
initial and final cube vertices are prescribed, then every marked path through
a sweep is forced.  Usually the $h$ cards ask for $hd$ independent switch
bits.  Whenever two marked cards occupy the two ends of the same active cube
edge, however, they ask for the \emph{same} switch bit.  Such an event is a
collision and saves one independent condition.  This gives the exact routing
weight
\[
 P_h(x,y)=n^{-h}2^{C(x,y)}
\]
for a feasible route, where $C(x,y)$ is its collision count.

For a fixed future sweep, pull every active matching edge back to the input
vertices.  The resulting collision graph $G$ is a simple $d$-regular graph,
and the number of collisions made by an occupied marked set $A$ is exactly
$e_G(A)$.  Its geometry satisfies the deterministic isoperimetric bound
\[
 e_G(B)\le \frac{|B|}{2}\log_2|B|.
\]
After one independent preliminary sweep, negative dependence gives
$\Pr(B\subseteq A)\le(h/n)^{|B|}$.  Expanding over connected edge sets then
controls the exponential moment of $e_G(A)$.  Interpolating row and column
$L^p$ bounds converts that moment estimate into the pointwise smoothing bound
\[
 N_hP_h^t(x,y)\le e^{\alpha h^2}.
\]
Here $N_h=(n)_h$ is the number of ordered distinct $h$-tuples.  Since uniform
motion on ordered tuples would have kernel $1/N_h$, the normalized quantity
$N_hP_h^t(x,y)$ is exactly the pointwise likelihood ratio relative to uniform.
Thus $e^{\alpha h^2}$ close to one means that every prescribed $h$-card
endpoint is nearly as likely as under uniformity.

\paragraph{2. A sharper entropy clock.}
Marked-card smoothing is useful only after the full-deck law is prevented from
putting too much mass on a tiny exceptional set of permutations.  Relative
entropy measures this global spikiness.  For an arbitrary input arrangement
$\mu$, expose one fresh sweep backwards in input order.  At each input $j$ the
unexposed endpoint is uniform on an explicit candidate set.  A symmetric
partner construction produces a random earlier input whose index is uniform
on $\{0,\ldots,j-1\}$ after averaging over the fresh sweep.

The construction yields two different lower bounds for the entropy loss
$\Delta$ in one sweep.  One detects square-root/Hellinger separation and the
other detects relative entropy from an averaged partner law.  They are
\emph{separate} lower bounds; the proof never adds them and therefore never
spends the same entropy loss twice.  The deterministic partner weights
$W_{jk}$ have the critical moment exponent
\[
 \beta_c=\log_2(3/2).
\]
Working just below this exponent and using an adjustable scalar interpolation
inequality gives the preparation time
\[
 R_*(\eta ne^{-s})=
 \frac{s^2}{2\log(3/2)}
 +O_\eta(s\log(s+2)+\log(d+2))
\]
in sweeps.

\paragraph{3. Ordered-tuple multiplicity amplification.}
This is the bridge from a few marked cards to the entire permutation.  Fourier
analysis on $S_n$ decomposes an arbitrary error from uniformity into
orthogonal representation levels.  A low-level representation occurs many
times inside the permutation action on ordered $h$-tuples.  For
$\lambda=(n-k,\mu)$ its exact multiplicity is
\[
 m_h(\lambda)=\binom hk f^\mu,
\]
while
$f^\lambda\le\binom nk f^\mu$.  Therefore a Hilbert--Schmidt budget for the
marked-card matrix forces small Fourier operators and energies on every low
level.  A sign-twisted tuple module gives the symmetric statement for
partitions close to a long column.  The use of \emph{ordered} tuples is
essential for this exact multiplicity.

\paragraph{4. Entropy truncation.}
Small entropy does not imply a small $L^2$ norm: a density can still have very
rare spikes.  We therefore cut off the high-density tail.  This is done as an
exact probability decomposition, not by changing the Markov chain:
\[
 F=pf+(1-p)b,\qquad p\ge1-\eta.
\]
The regular part $f$ has controlled $L^2$ norm and obeys the pointwise
domination $f\le F/p$.  Hence every marked-card estimate already proved for
the actual block $F$ passes to $f$ at the cost of only the factor $p^{-1}$.
The exceptional component has total mass at most $\eta$ and will simply be
paid for in total variation.

\paragraph{5. Two long regular blocks and one short actual block.}
Two independent copies of the regular component $f$ are enough to suppress
high representation levels.  A third block is needed only to contract the
remaining low levels, so that block need not pay the entropy-preparation cost.
Let its actual density be $G$.  The regular part analyzed at the end is
\[
 g=G*f*f.
\]
If the long-block marked parameter is $\alpha$ and the short-block parameter
is $\beta_0$, the low-level geometric ratio is essentially
$n\alpha\sqrt{\beta_0K}$.  High levels are small because the two $f$ factors
have controlled $L^2$ norm and high representation dimensions are large.

The scale of the final exponent can already be guessed before optimizing
constants.  To reach entropy roughly $ne^{-s}$ costs order $s^2$ sweeps,
whereas the marked-card smoothing needed at that scale costs order $d/s$
sweeps.  Balancing
\[
 s^2\asymp d/s
\]
gives $s\asymp d^{1/3}$ and a sweep count of order $d^{2/3}$.  Since one sweep
is $d$ physical Thorp shuffles by \eqref{eq:clock-conversion}, the physical
time is of order $d^{5/3}$.  The quantitative optimization in
Section~\ref{sec:three-block} and the following section produces
$z_*=(17+\sqrt{97})/8$ and the explicit constant $C_*$.

The next sections carry this program out quantitatively.  The longer
shuffle-specific proofs are included after the main argument so that every
asserted input can be reconstructed from this manuscript.

\section{Quantitative inputs used in the proof}
The first input is the precise marked-card statement produced by the collision
analysis.  Keep the following interpretation in mind.  Under the uniform law
on ordered $h$-tuples, each endpoint has probability $1/N_h$.  Hence
$N_hP_h^t(x,y)$ is the ratio
\[
 \frac{\Pr\{\text{the $h$ marked cards end at }y\mid x\}}
      {\Pr\{\text{the same endpoint under a uniform permutation}\}}.
\]
A bound by $e^{\alpha h^2}$ says that correlations among the $h$ marked cards
can inflate any endpoint probability by at most that factor.

For ordered distinct $h$-tuples, let $N_h=(n)_h$ and $P_h$ be the
one-sweep transition matrix. For an integer $m\ge2$, define
\begin{equation}\label{eq:alpha}
 \alpha_{d,m}={8Ld\over n}(2n)^{1/[2(m-1)]}.
\end{equation}
\begin{corollary}\label{cor:quadratic}
If $\alpha_{d,m}h\le m-1$, then, for every integer $t\ge2m$,
\begin{equation}\label{eq:quadratic-budget}
 N_hP_h^t(x,y)\le \exp(\alpha_{d,m}h^2).
\end{equation}
The same assertion holds for reverse-coordinate sweeps.
\end{corollary}
\begin{proof}
The smoothing theorem in Appendix~\ref{sec:kernel} has
$a=4d(h/n)(1-2^{-1/(m-1)})(2h)^{1/[2(m-1)]}$ and upper exponent
$(m-1)ha/(1-a)$. Since
$(m-1)(1-2^{-1/(m-1)})\le L$, one has
$a\le\alpha_{d,m}h/[2(m-1)]\le1/2$. The exponent is therefore at most
$\alpha_{d,m}h^2$. Additional stochastic transitions preserve the maximum.
The $h=1$ case also follows directly from one-card uniformity.
\end{proof}
The coarse entropy profile proved in Appendix~\ref{sec:entropy} is
\begin{equation}\label{eq:coarse-profile}
 {E\over\Delta}\le100+12\log^+(n/E),\qquad E>0.
\end{equation}
We use this only to bootstrap a logarithm below, not for the leading
coefficient of the new clock.

\section{An entropy clock at the critical endpoint-moment threshold}
Before entering the calculation, we summarize the objects that appear in it.
Let $\mu$ be the current random deck arrangement and let one fresh independent
sweep act on it.  Write
\[
 E=\Ent(\mu),\qquad
 \Delta=\Ent(\mu)-\Ent(\text{deck after the fresh sweep}).
\]
Thus $E$ is the remaining information relative to a uniform deck and $\Delta$
is the amount erased by one sweep.

The detailed exposure argument in Appendix~\ref{sec:information} reveals
positions from high index to low index.  When $j+1$ labels remain unrevealed,
$p_j$ is the conditional distribution of the label occupying position $j$.
The partner construction supplies deterministic weights $W_{jk}$ and the
comparison distribution
\[
 q_j=\sum_{k\le j}W_{jk}p_k.
\]
Two aggregate discrepancies are used:
\[
 S=\sum_j\E\!\left[1-\frac1{j+1}
       \left(\sum_a\sqrt{p_j(a)}\right)^2\right],
 \qquad
 \mathcal C=\sum_j\E D(p_j\Vert q_j).
\]
The first is a square-root/Hellinger-type defect from conditional uniformity;
the second measures how distinguishable position $j$ is from its randomized
partner.  Appendix~\ref{sec:information} proves separately that
$\Delta\ge(\log2)S$ and $\Delta\ge\mathcal C$.

In this section the notation $p_j,q_j,W_{jk},S,\mathcal C$ is exactly that
of Appendix~\ref{sec:information}. Its information inequalities give
\begin{equation}\label{eq:v4-losses}
 E=\sum_j\E D(p_j\Vert u_{j+1}),\qquad
 \mathcal C\le\Delta,\qquad S\le\Delta/L.
\end{equation}
Put $\beta_c=\log_2(3/2)$.

\begin{lemma}[General endpoint moment]\label{lem:beta-moment}
For $0<\beta<\beta_c$, set
\[
 R_j^{(\beta)}=(j+1)^\beta\sum_{k\le j}W_{jk}^{1+\beta},\qquad
 C_\beta={1\over3-2^{1+\beta}}.
\]
On every conditioning atom,
\[
 (j+1)^\beta\sum_a q_j(a)^{1+\beta}\le R_j^{(\beta)},\qquad
 1\le R_j^{(\beta)}\le(j+1)^\beta,
 \qquad \sum_{j=1}^{n-1}R_j^{(\beta)}\le C_\beta n.
\]
\end{lemma}
\begin{proof}
The matrix $(p_k(a))$ is doubly stochastic. Jensen followed by summation
over $a$ bounds $\sum_a(\sum_kW_{jk}p_k(a))^{1+\beta}$ by
$\sum_kW_{jk}^{1+\beta}$. The pointwise bounds follow because $W_j$ is a
probability vector.

Let $U_j=\sum_{k\le j}W_{jk}^{1+\beta}$ and
$T_d=\sum_{j<2^d}U_j$. The exact weight recursion in
\eqref{eq:W} gives
\[
 T_d=(1+2^{-1-\beta})T_{d-1}
             +2^{-1-\beta}2^{(1-\beta)(d-1)},\qquad T_0=1.
\]
For $Z_d=T_d/2^{(1-\beta)d}$ this becomes
\[
 Z_d=(2^{\beta-1}+1/4)Z_{d-1}+1/4.
\]
Its multiplier is below one precisely when $\beta<\beta_c$; its fixed
point is $C_\beta\ge1$. Therefore $Z_d\le C_\beta$. Finally
$\sum_jR_j^{(\beta)}\le n^\beta T_d\le C_\beta n$.
\end{proof}

\begin{lemma}[Adjustable scalar interpolation]\label{lem:v4-scalar}
Let $p,q$ be distributions on $\ell$ points and suppose
$\ell^\beta\sum_aq(a)^{1+\beta}\le R$ with $R\ge1$.
For $0<\delta<1$ and $\beta>0$ put
\[
 c_\beta=1+1/\beta,\quad
 a_{\beta,\delta}={1\over\beta(1-\delta)},\quad
 b_{\beta,\delta}=\log(8/\delta)+2+c_\beta.
\]
Writing $V=1-\ell^{-1}(\sum_a\sqrt{p(a)})^2$, one has
\begin{equation}\label{eq:v4-scalar}
 D(p\Vert u_\ell)\le c_\beta D(p\Vert q)
       +[a_{\beta,\delta}\log(R/V)+b_{\beta,\delta}]V.
\end{equation}
The last term is defined to be zero at $V=0$.
\end{lemma}
\begin{proof}
If $D(p\Vert q)=\infty$ there is nothing to prove. Set $f=\ell p$,
$A_0=\E_{u_\ell}\sqrt f=\sqrt{1-V}$ and $h(x)=x\log x-x+1$.
As in Appendix~\ref{sec:scalar},
\[
 D(p\Vert u_\ell)\le\E_{u_\ell} A_0^2h(f/A_0^2).
\]
Suppose first that $0<V\le1-\delta$, so $A_0^2\ge\delta$.
Choose $M=8(R/V)^{1/\beta}$. The elementary scalar truncation
Lemma~\ref{lem:scalar-trunc}, with $T=M/A_0^2\ge4$, gives
\[
 D(p\Vert u_\ell)\le(\log M-\log A_0^2+2)V+T_1,
 \quad T_1=\sum_{f(a)>M}p(a)\log(f(a)/M).
\]
Let $\theta=\beta/(1+\beta)$. The entropy variational inequality gives
\[
 \theta T_1\le D(p\Vert q)+\log\E_q
                 \exp\{\theta\log^+(f/M)\}.
\]
H\"older yields
\[
 \E_q\exp\{\theta\log^+(f/M)\}
 \le1+M^{-\theta}R^{1/(1+\beta)}p\{f>M\}^{\theta}.
\]
For $f>M\ge8$, $f\le4(\sqrt f-1)^2$. Hence
$p\{f>M\}\le8(1-A_0)\le8V$. The added term is at most $V$ by the
choice of $M$. Thus $T_1\le c_\beta D(p\Vert q)+c_\beta V$.
Substitution proves \eqref{eq:v4-scalar}, even with coefficient
$1/\beta$ instead of $a_{\beta,\delta}$ in this case.

For $V>1-\delta$, apply the entropy variational inequality to
$\beta\log(\ell q)$, obtaining
\[
 D(p\Vert u_\ell)\le c_\beta D(p\Vert q)+\beta^{-1}\log R.
\]
Since $V/(1-\delta)>1$ and $\log(R/V)\ge\log R\ge0$, this is bounded by
\eqref{eq:v4-scalar}. The uniform case and zeros follow by continuity.
\end{proof}

\begin{proposition}[Explicit improved clock]\label{prop:v4-clock}
For $1\le B\le n$, put
\[
 J=\left\lceil{\log(n/B)\over L}\right\rceil,\quad Z=LJ,\quad
 \delta={1\over Z+4},\quad \beta=\beta_c-{1\over Z+4}.
\]
Use the constants of Lemma~\ref{lem:v4-scalar}, and define
\begin{align}
 D_0&={a_{\beta,\delta}\over L},\nonumber\\
 F_0&=c_\beta+{b_{\beta,\delta}\over L}
   +D_0\left[\log^+(C_\beta L)+\log(100+12Z)+{1\over e}\right],\nonumber\\
 R_*(B)&=\left\lceil100\log^+(dL)\right\rceil
       +\sum_{j=1}^J\left\lceil L(F_0+D_0jL)\right\rceil.
 \label{eq:v4-clock}
\end{align}
Then $\Ent(q_{R_*(B)})\le B$. For fixed $\eta>0$ and
$B=\eta ne^{-s}\in[1,n]$, as $s\to\infty$,
\begin{equation}\label{eq:v4-clock-asymptotic}
 R_*(B)=A s^2+O_\eta(s\log(s+2)+\log(d+2)),
 \qquad A={1\over2\log(3/2)}.
\end{equation}
\end{proposition}
\begin{proof}
Summing \eqref{eq:v4-scalar}, using the endpoint moment and log-sum
inequality (also over conditioning atoms), gives
\[
 E\le c_\beta\mathcal C+
 a_{\beta,\delta}S\log(C_\beta n/S)+b_{\beta,\delta}S.
\]
Let $x=E/\Delta\ge1$. For $v=LS/\Delta\in(0,1]$, use
$v\log(1/v)\le1/e$ to obtain
\[
 S\log(C_\beta n/S)
 \le{\Delta\over L}[\log^+(C_\beta Ln/\Delta)+1/e].
\]
Consequently, with $r=\log^+(n/E)$,
\[
 x\le c_\beta+b_{\beta,\delta}/L+
 D_0[\log^+(C_\beta L)+r+\log x+1/e].
\]
When $E\ge ne^{-Z}$, the coarse profile
\eqref{eq:coarse-profile} bounds $\log x\le\log(100+12Z)$. Thus
$x\le F_0+D_0r$. While entropy is in the $j$th dyadic band, this is at
most $F_0+D_0jL$, so the corresponding summand of \eqref{eq:v4-clock}
suffices to halve the previous threshold. The initial term reduces entropy
from at most $ndL$ to $n$ using \eqref{eq:coarse-profile}. This proves
the finite assertion, including $J=0$.

The chosen $\beta$ is bounded away from zero, and
\[
 D_0={1\over\beta_c L}+O((Z+4)^{-1}),\qquad
 C_\beta=O(Z+4),\qquad F_0=O(\log(Z+2)).
\]
Summing the ceilings gives
$R_*(B)=\frac12D_0Z^2+O(Z\log(Z+2)+\log(d+2))$.
Since $\beta_cL=\log(3/2)$ and $Z=s+\log(1/\eta)+O(1)$, this is
\eqref{eq:v4-clock-asymptotic}. All constants are uniform as the moment
order approaches its critical endpoint at the stated rate.
\end{proof}

\section{Ordered-tuple multiplicity amplification}\label{sec:tuple-amplification}
This section is the algebraic bridge from ``every small ordered collection of
cards is smooth'' to ``the full permutation is smooth.''  A reader unfamiliar
with representation theory may read Lemma~\ref{lem:HS-amplification} as the
bridge theorem and then inspect its proof for the bookkeeping that produces
the binomial factors.

The regular representation of $S_n$ decomposes functions on $S_n$ into
orthogonal matrix-valued Fourier modes indexed by partitions $\lambda\vdash n$.
Tracking an ordered $h$-tuple of distinct labels gives another representation:
$S_n$ acts by permuting the set $\Omega_h$.  The key fact is that a low-level
Fourier mode occurs with large, explicitly computable multiplicity inside this
ordered-tuple action.  This is why information about only $h$ cards can control
a much larger part of the full-deck distribution.

A \emph{standard Young tableau} of a (possibly skew) Young diagram is a filling
of its boxes by consecutive integers, increasing from left to right in every
row and from top to bottom in every column.  The branching rule identifies the
multiplicity below with the number of such tableaux of an appropriate skew
shape.

All statements in this section concern arbitrary probability laws on
$S_n$, not just the Thorp walk. For a partition $\lambda\vdash n$, let
$V_\lambda$ be its irreducible representation and $f^\lambda=\dim V_\lambda$.
Its level is
\[
 \ell(\lambda)=\min(n-\lambda_1,n-\lambda'_1).
\]
Let $E_k$ be the orthogonal projection in $L^2(U)$ onto the sum of all
isotypic components of level $k$, and let $E_{>K}$ include every level
larger than $K$. Level zero consists of the constant and sign characters.
For convolution by a probability density $f$, write $T_f$ for the
operator and $\norm{T_f}_k$ for its norm on level $k$.

For an orthogonal realization $\rho_\lambda$, put
\[
 \widehat f(\lambda)=\E_{g\sim fU}\rho_\lambda(g).
\]
Plancherel and the induced action on ordered tuples give
\begin{align}
 \norm{E_k f}_2^2
 &=\sum_{\ell(\lambda)=k}f^\lambda
                 \norm{\widehat f(\lambda)}_\HS^2,\label{eq:plancherel}\\
 \norm{P_h^f}_\HS^2
 &=\sum_{\lambda\vdash n}m_h(\lambda)
                 \norm{\widehat f(\lambda)}_\HS^2.
 \label{eq:induced-HS}
\end{align}
Here $P_h^f$ is the stochastic tuple transition matrix and $m_h(\lambda)$
its representation multiplicity. Hilbert--Schmidt norm in
\eqref{eq:induced-HS} is the ordinary matrix norm
$\sum_{x,y}(P_h^f(x,y))^2$; no additional $N_h$ factor is present.

\begin{lemma}[Exact multiplicity and dimension ratio]\label{lem:multiplicity}
Let $1\le k\le h\le n-k$ and $\lambda=(n-k,\mu)$ with $\mu\vdash k$.
Then
\begin{equation}\label{eq:multiplicity}
 m_h(\lambda)=\binom hk f^\mu,
 \qquad f^\lambda\le\binom nk f^\mu.
\end{equation}
\end{lemma}
\begin{proof}
The ordered-tuple module is $\operatorname{Ind}_{S_{n-h}}^{S_n}\mathbf1$.
By reciprocity and iterated Young branching, $m_h(\lambda)$ is the number
of standard tableaux of the skew shape $\lambda/(n-h)$; the branching
rule is recalled in \cite[Section 2.2]{KeevashLifshitz2023}. Because $n-h\ge k\ge\mu_1$,
this skew shape is a separated row of length $h-k$ and the lower diagram
$\mu$. Choose the $k$ entries placed in $\mu$ and fill it in $f^\mu$ ways;
the other row is forced. This proves the equality. In a full standard
tableau of shape $\lambda$, choosing the $k$ lower entries and a tableau
of shape $\mu$ determines the first row, but some choices may violate
column inequalities. This gives the upper bound.
\end{proof}

\begin{lemma}[Unsigned and signed marked bounds]\label{lem:HS-amplification}
Let $f$ be a probability density and
$\mathcal B_h=\norm{P_h^f}_\HS^2$. Suppose
$1\le k<(n-1)/2$ and $k\le h\le n-k$. Then
\begin{align}
 \norm{T_f}_k^2
 &\le\frac{\mathcal B_h}{\binom hk},\label{eq:op-amplification}\\
 \norm{E_k f}_2^2
 &\le2\mathcal B_h\frac{\binom nk}{\binom hk},\label{eq:energy-amplification}\\
 \norm{E_k(f*f*f)}_2^2
 &\le2\mathcal B_h^3\frac{\binom nk}{\binom hk^3}.
 \label{eq:cubic-amplification}
\end{align}
\end{lemma}
\begin{proof}
For the row wing $\lambda_1=n-k$, \eqref{eq:induced-HS} and
\eqref{eq:multiplicity} imply
\[
 \norm{\widehat f(\lambda)}_{\rm op}^2
 \le\norm{\widehat f(\lambda)}_\HS^2
 \le\frac{\mathcal B_h}{\binom hk f^\mu}
 \le\frac{\mathcal B_h}{\binom hk}.
\]
Also $f^\lambda/m_h(\lambda)\le\binom nk/\binom hk$, so summing the
nonnegative terms in \eqref{eq:induced-HS} proves the energy bound for
that wing without any factor counting partitions.

For the column wing, use the sign-twisted tuple module. Its matrix has
entries
\[
 \sum_g f(g)U(g)\sgn(g)\ind_{g x=y},
\]
whose absolute values are at most $P_h^f(x,y)$. Its squared Hilbert--Schmidt norm is therefore at
most $\mathcal B_h$. The multiplicity of $\lambda$ in this module is the
unsigned multiplicity of $\lambda'$. This proves the same bounds for
the column wing. The two wings are disjoint in the stated range, giving
the factor two in \eqref{eq:energy-amplification}.
Finally matrix multiplication and
$\norm{ABC}_\HS\le\norm A_\HS\norm B_{\rm op}\norm C_{\rm op}$ give
\eqref{eq:cubic-amplification}. No matrix normality or eigenvalue estimate
is used.
\end{proof}

\section{Truncation preserves the marked budget}
The purpose of this section is elementary but important.  Relative entropy
controls an average of $F\log F$, whereas the high-level representation bound
needs an $L^2$ norm.  A tiny set on which $F$ is enormous can have modest
entropy but terrible $L^2$ norm.  We isolate that rare set exactly.

Choose a threshold $M$.  The mass of the set $\{F>M\}$ is at most $\eta$ once
$M$ is chosen from the entropy bound.  Conditioning on the complementary event
produces a probability density $f$ with $f\le F/p$ and controlled $L^2$ norm;
the removed mass becomes another probability density $b$.  Thus
$F=pf+(1-p)b$ is an identity of probability laws.  Nothing is discarded from
the actual shuffle, and the exceptional probability $1-p$ is tracked all the
way to the final total-variation estimate.

\begin{lemma}[Entropy truncation with pointwise domination]\label{lem:new-truncate}
Let $F$ be a probability density with $\Ent(F)\le B$, and let
$0<\eta<1$. Set $\log M=(B+1/e)/\eta$. There is an exact decomposition
\begin{equation}\label{eq:new-mixture}
 F=p f+(1-p)b,\qquad p\ge1-\eta,
\end{equation}
into probability densities, with
\begin{equation}\label{eq:new-truncate}
 f\le F/p,\qquad
 \log\norm f_2\le\frac{B+1/e}{2\eta}
                     +\frac12\log\frac1{1-\eta}.
\end{equation}
If $N_hP_h^F(x,y)\le e^{\alpha h^2}$, then
\begin{equation}\label{eq:truncated-HS}
 \norm{P_h^f}_\HS^2\le p^{-1}e^{\alpha h^2}.
\end{equation}
\end{lemma}
\begin{proof}
Let $a=\E_U[F\ind_{F>M}]$. Since $x\log x\ge-1/e$,
$a\log M\le\Ent(F)+1/e\le B+1/e$, so $a\le\eta$.
Set $p=1-a$, $f=F\ind_{F\le M}/p$, and normalize the tail as $b$;
when $a=0$, the choice of $b$ is immaterial. Then
$\norm f_2^2\le M/p$, proving \eqref{eq:new-truncate}.
Domination passes to every tuple transition entry. Because $P_h^f$ is
stochastic,
\[
 \sum_{x,y}P_h^f(x,y)^2
 \le\max_{x,y}P_h^f(x,y)\sum_{x,y}P_h^f(x,y)
 \le \frac{e^{\alpha h^2}}{pN_h}N_h.
\]
This proves \eqref{eq:truncated-HS}.
\end{proof}

The last inequality is why entropy preparation and marked smoothing can
be performed by the \emph{same} actual block. It is not an assertion that
one may overlap independent factors without paying their time.

\begin{proposition}[Optimized marked order]\label{prop:optimized-order}
Suppose $t\ge2m$, $F=q_t$, and $f$ is as in
Lemma~\ref{lem:new-truncate}. Let $\alpha$ be \eqref{eq:alpha} and assume
\begin{equation}\label{eq:order-domain}
 \alpha n\ge16,\qquad 1\le K<n/10,\qquad \alpha K\le1/16.
\end{equation}
Then, for $1\le k\le K$,
\begin{align}
 \norm{T_f}_k^2
 &\le p^{-1}\bigl(2e\sqrt{\alpha k}\bigr)^k,
 \label{eq:optimized-op}\\
 \norm{E_k f}_2^2
 &\le p^{-1}\bigl(4en\sqrt{\alpha/k}\bigr)^k,
 \label{eq:optimized-energy}\\
 \norm{E_k f^{*3}}_2^2
 &\le p^{-3}\bigl(16e^3n\alpha^{3/2}\sqrt{k}\bigr)^k.
 \label{eq:optimized-cube}
\end{align}
\end{proposition}
\begin{proof}
For this level choose $h=\lfloor\sqrt{k/\alpha}\rfloor$ and put
$r=\sqrt{k/\alpha}$. Then $r\ge4k$,
$h\ge r/2\ge k$, and $h-k+1\ge r-k\ge r/2$.
Also $h\le\sqrt{n/\alpha}\le n/4\le n-k$ and
$\alpha h\le\sqrt{\alpha k}\le1/4\le m-1$. Corollary~\ref{cor:quadratic}
and truncation give $\mathcal B_h\le p^{-1}e^{\alpha h^2}\le p^{-1}e^k$.
Use
\[
 \binom hk\ge(h/k)^k,\qquad
 \frac{\binom nk}{\binom hk}
 =\prod_{j=0}^{k-1}\frac{n-j}{h-j}
 \le\left(\frac n{h-k+1}\right)^k.
\]
Inserting these in Lemma~\ref{lem:HS-amplification} gives
\eqref{eq:optimized-op} and \eqref{eq:optimized-energy}, with
$2\le2^k$ absorbing the two wings. Their product gives
\eqref{eq:optimized-cube}. The marked order is allowed to depend on $k$;
all estimates concern the same density $f$.
\end{proof}

\section{Two regular blocks and one actual smoothing block}\label{sec:three-block}
We now assemble the proof using actual consecutive shuffle blocks.  It is
useful to keep the roles distinct:
\begin{itemize}[leftmargin=2em]
\item a \emph{long block} of length $t$ sweeps is run long enough both to lower
its entropy and to smooth every marked-card order needed later;
\item from each long block we analyze only its high-probability regular
component $f$, while paying at most $\eta$ for the exceptional component;
\item a \emph{short block} of length $u$ sweeps is kept as its actual density
$G=q_u$.  It is not truncated and need not satisfy an entropy target.  Its
only job is to provide one more low-level contraction.
\end{itemize}
Under our convolution convention the regular component to estimate is
$g=G*f*f$.  Because the original blocks are powers of the same sweep kernel,
the corresponding actual law is exactly $q_u*q_t*q_t=q_{2t+u}$.  The proof
below never reorders $G,f,f$ after the two truncations.

The parameters in Theorem~\ref{thm:v4-finite} have the following roles.
The number $H$ is a logarithmic $L^2$ budget for $f$; $K=\lfloor H\rfloor$ is
the Fourier level at which we switch from ordered-tuple control to the
high-level dimension bound.  The quantities $\alpha$ and $\beta_0$ measure the
remaining marked-card interaction in the long and short blocks, respectively.
Finally
$\rho=16e^3n\alpha\sqrt{\beta_0K}$ is the common ratio of the low-level
geometric series.  The theorem simply requires $\rho<1$ and makes every error
term explicit.

The high-level dimension estimate is
\begin{equation}\label{eq:high-dimension}
 \norm{T_f}_{>K}\le\norm f_2(eK/n)^{K/2},\qquad1\le K<n/10.
\end{equation}
This is the external estimate \eqref{eq:external-high-level}; see \cite[Theorem 4.4 and Lemma 4.5]{KeevashLifshitz2023}.
Every probability convolution contracts every central isotypic level.

\begin{theorem}[Asymmetric finite criterion]\label{thm:v4-finite}
Choose $0<\eta<1/3$, $H\ge2$, $B=\eta H\in[1,n]$, and
\[
 H\ge{1\over e\eta}+\log{1\over1-\eta},\qquad K=\lfloor H\rfloor<n/10.
\]
Let $m_a,m_b\ge2$ be integers, $\alpha=\alpha_{d,m_a}$,
$\beta_0=\alpha_{d,m_b}$, and
\[
 t=\max\{R_*(B),2m_a\},\qquad u=2m_b.
\]
Suppose $\alpha n,\beta_0n\ge16$, $\alpha K,\beta_0K\le1/16$, and
\[
 \rho=16e^3n\alpha\sqrt{\beta_0K}<1.
\]
Then
\begin{equation}\label{eq:v4-finite-TV}
 \norm{q_{2t+u}U-U}_{\TV}\le2\eta+\frac12\left[
 (1-\eta)^{-2}{\rho\over1-\rho}
                 +e^{4H}(eK/n)^K\right]^{1/2}.
\end{equation}
The physical time is exactly $d(2t+u)$ ordinary Thorp shuffles.
\end{theorem}
\begin{proof}
Let $F=q_t$. The improved clock gives $\Ent(F)\le B$, and the marked
kernel bound holds simultaneously at every required order.
Lemma~\ref{lem:new-truncate} gives $F=pf+(1-p)b$, $p\ge1-\eta$, with
$\log\norm f_2\le H$. Two independent copies have regular component
$f*f$ with weight $p^2$; the other components have mass at most $2\eta$.
Put $G=q_u$, without truncating or imposing an entropy requirement on it.

The proof of Proposition~\ref{prop:optimized-order}, applied to $G$ with
$p=1$ and parameter $\beta_0$, gives
\[
 \norm{T_G}_k^2\le(2e\sqrt{\beta_0 k})^k,\qquad1\le k\le K.
\]
This uses only the marked kernel estimate, not an entropy hypothesis.
For the regular component $g=G*f*f$, the same proposition gives
\begin{align*}
 \norm{E_k g}_2^2
 &\le\norm{T_G}_k^2\norm{T_f}_k^2\norm{E_kf}_2^2\\
 &\le p^{-2}(16e^3n\alpha\sqrt{\beta_0 k})^k.
\end{align*}
Summation over $k\le K$ gives the first term in the square brackets of
\eqref{eq:v4-finite-TV}. For high levels, probability contraction and
\eqref{eq:high-dimension} yield
\[
 \norm{E_{>K}g}_2^2\le\norm{E_{>K}(f*f)}_2^2
           \le\norm f_2^4(eK/n)^K\le e^{4H}(eK/n)^K.
\]
The actual factor $G=q_u$ has sign coefficient zero, because it includes
at least one fair-switch sweep. Therefore $g$ has only its constant
coefficient on level zero. No extra parity-removal sweep is needed.
Orthogonality, $\TV\le\tfrac12L^2$, and the exact mixture prove the
bound. The actual law is $q_u*q_t*q_t=q_{2t+u}$. No factors were
reordered after regularization, and no artificial kernel replaced an
actual block.
\end{proof}

\section{Choice of parameters and optimization of the coefficient}
We first explain the scale before choosing the constants.  Put
$H\approx ne^{-s}$.  Reaching entropy of this order costs
approximately $As^2$ sweeps, where $A=1/(2\log(3/2))$.  On the other hand,
the marked-card bound contains an exponential factor of the form
$\exp\{c d/m\}$, so making it compatible with the cutoff $H\approx ne^{-s}$
requires a smoothing depth $m$ of order $d/s$.  Hence the long-block duration
is essentially
\[
 \max\{As^2,\;c d/s\}.
\]
Balancing the two mechanisms forces $s\asymp d^{1/3}$ and a block length
$\asymp d^{2/3}$ sweeps.  The remaining parameter $z>2$ controls how the
smoothing burden is divided between the two long blocks and the shorter third
block.  The calculation below chooses the value of $z$ that minimizes the
leading constant.

Fix $\eta=\varepsilon/8$. For sufficiently large $d$, set
\begin{align}
 v&=(z_*Ld/A)^{1/3},&
 s&=v+4\log(d+2)+100,\nonumber\\
 H&=ne^{-s},&B&=\eta H,\qquad K=\lfloor H\rfloor,\nonumber\\
 m_a&=2+\left\lceil{z_*L(d+1)\over2v}\right\rceil,&
 m_b&=2+\left\lceil{z_*L(d+1)\over2(z_*-2)v}\right\rceil.
 \label{eq:v4-parameters}
\end{align}
Let $t,u$ be as in Theorem~\ref{thm:v4-finite}. Because $s=o(d)$, all
entropy and norm-budget domains eventually hold. The two marked parameters
satisfy
\[
 \alpha\le{8Ld\over n}e^{v/z_*},\qquad
 \beta_0\le{8Ld\over n}e^{(z_*-2)v/z_*}.
\]
Both are at least $8Ld/n$, and both multiplied by $K$ tend to zero.
The low-level ratio is bounded by
\begin{align*}
 \rho&\le16e^3(8Ld)^{3/2}
  \exp\left\{{v\over z_*}+{(z_*-2)v\over2z_*}-{s\over2}\right\}\\
 &=16e^3(8Ld)^{3/2}e^{(v-s)/2}=O(d^{-1/2}).
\end{align*}
For the high-level term, $H/2\le K\le H$ eventually, so
\[
 4H+K\log(eK/n)\le H(4.5-s/2)\longrightarrow-\infty.
\]
The right side of \eqref{eq:v4-finite-TV} is thus eventually less than
$\varepsilon$.

The times satisfy
\begin{align*}
 R_*(B)&=Av^2+O_\varepsilon(v\log(d+2)+\log^2(d+2)),\\
 2m_a&={z_*Ld\over v}+O(1)=Av^2+O(1),\\
 u=2m_b&={z_*Ld\over(z_*-2)v}+O(1)
                          ={Av^2\over z_*-2}+O(1).
\end{align*}
Taking the genuine maximum for $t$ and adding both long blocks and the
short block gives
\[
 d(2t+u)=\left(2+{1\over z_*-2}\right)
   A^{1/3}(z_*L)^{2/3}d^{5/3}
                  +O_\varepsilon(d^{4/3}\log(d+2)).
\]
This proves Theorem~\ref{thm:main-intro}. Group translation gives every
initial deck. Finitely many dimensions are absorbed into the error-term
constant: identity has positive sweep probability and hypercube-edge
transpositions belong to sweep support and generate $S_n$.

\paragraph{Why this particular algebraic number?}
Within this two-equal-regular-block, one-actual-block family, any fixed
$z>2$ yields the coefficient
\[
 A^{1/3}L^{2/3}F(z),\qquad F(z)=z^{2/3}\left(2+{1\over z-2}\right).
\]
One has
\[
 F'(z)={4z^2-17z+12\over3z^{1/3}(z-2)^2}.
\]
Since $F$ diverges at both ends of $(2,\infty)$, its unique minimizer is
$z_*=(17+\sqrt{97})/8$. This is optimization only within this stated
family, not an optimality assertion for Thorp mixing. The actual shorter
block has asymptotic length $1/(z_*-2)=0.7374048168\ldots$ times that of
each regular block.

\section{Pitfalls, limitations, and the remaining gap}\label{sec:limitations}
Two shortcuts that might appear to improve the argument are invalid and are
worth recording.  First, unordered subsets cannot replace ordered tuples in
Section~\ref{sec:tuple-amplification}; the exact multiplicity
$\binom hk f^\mu$ is an ordered-tuple statement.  Second, a uniform two-sweep
pointwise estimate of the form $n!q_2(\sigma)\le e^{Cn}$ with an absolute
constant $C$ is false.  To see the latter, choose a matching of $n/8$ cube
edges and activate exactly those switches in both sweeps.  Greedy counting
gives at least
\[
 \frac{(nd/4)^{n/8}}{(n/8)!}\ge (2d)^{n/8}
\]
such choices.  Each selected edge is traversed twice and returns its two cards,
so the resulting permutation is the identity.  Since two sweeps have $n^n$
equiprobable switch configurations,
\[
 \log(n!q_2(\id))\ge \frac n8\log(2d)-n.
\]
This obstruction is only pointwise; it neither rules out an $O(n)$ entropy
bound after two sweeps nor yields a $d^{5/3}$ lower bound.

The method still balances a quadratic entropy-preparation cost in a parameter
$s$ against a marked-card smoothing cost of order $d/s$ sweeps.  Their balance
produces the power $d^{2/3}$ in sweep count, hence $d^{5/3}$ ordinary Thorp
shuffles.  Improving the exponent therefore appears to require a genuinely
stronger entropy clock, a stronger marked-card smoothing estimate, or a new
finishing mechanism.  No cutoff statement, limiting profile, or matching lower
bound is proved here.

\appendix
\section*{Detailed analytic inputs}
\addcontentsline{toc}{section}{Detailed analytic inputs}
For completeness, we now prove the shuffle-specific analytic statements used
in the main argument.  The only external negative-dependence input is (E1) in
Section~\ref{subsec:external}; all remaining estimates below are derived
explicitly.

\section{Collision graphs and inclusion probabilities}
\label{sec:graph}
A sweep is a network of $d$ matching layers.  Once all switch bits in that
network are fixed, every input vertex follows one deterministic path through
the network.  Pulling each physical matching edge backwards along those paths
creates a graph on the \emph{input} vertices.  Two marked input cards collide
in the sweep exactly when their two input vertices form an edge of this pulled-
back graph.  The next lemmas formalize this picture.

Let $\Omega_k$ be the ordered distinct $k$-tuples of vertices, let
$N_k=(n)_k=n(n-1)\cdots(n-k+1)$, and give $\Omega_k$ uniform probability
$\pi_k$. Let $P_k$ be the transition matrix of one sweep on these tuples.

\begin{lemma}[Exact labelled routing weight]\label{lem:routing}
For $x,y\in\Omega_k$, the one-sweep route is either infeasible or its
collision count $C(x,y)$ is determined by the endpoints and
\begin{equation}\label{eq:routing}
 P_k(x,y)=2^{-kd+C(x,y)}=n^{-k}2^{C(x,y)}.
\end{equation}
A collision is a matching edge containing two marked cards immediately
before its layer.
\end{lemma}
\begin{proof}
Before layer $i$, a card with initial vertex $x_a$ and terminal vertex $y_a$
must be at $(y_{a,<i},x_{a,\ge i})$. Thus every marked path is prescribed.
An inconsistent prescription is infeasible. Otherwise it prescribes one
fair switch for each visited edge. If $c_i$ edges are doubly occupied at
layer $i$, there are $k-c_i$ visited edges at that layer. Consequently the
number of independent prescribed switches is $kd-\sum_i c_i$. Unvisited
switches impose no condition.
\end{proof}

Fix a realization $\omega$ of all switches in a sweep. Let $T_i$ map input
vertices to their positions immediately before layer $i$. Pull that layer's
matching back by $T_i^{-1}$ and take the union over $i$; denote the resulting
graph by $G_\omega$.

\begin{lemma}[Uniform collision-graph geometry]\label{lem:graph}
For every $\omega$, the graph $G_\omega$ is simple and $d$-regular. The number
of collisions of a marked input set $A$ is exactly $e_{G_\omega}(A)$.
Every nonempty vertex set $B$ satisfies
\begin{equation}\label{eq:iso}
 e_{G_\omega}(B)\le\frac{|B|}{2}\log_2|B|.
\end{equation}
The same statements hold in the reverse coordinate order.
\end{lemma}
\begin{proof}
Every pulled-back matching has degree one. Endpoints of a layer-$i$ edge
have highest differing initial coordinate $i$, because previous layers
changed only lower coordinates. Thus edges from distinct layers do not
repeat. The collision identity is the definition of pullback.

Fix the suffix above coordinate $i$. Let $a,b$ be the numbers of vertices of
$B$ having that suffix and having bit $i$ equal to zero or one. Earlier
layers do not change these counts, so at most $\min(a,b)$ of the layer-$i$
matching edges connect two such vertices. For binary entropy $h_2$,
\[
 \min(a,b)\le\frac{a+b}{2}h_2\!\left(\frac a{a+b}\right),
\]
with the empty-node contribution defined as zero. This follows from
$h_2(p)\ge2\min(p,1-p)$. Summing over the suffix tree telescopes to
$|B|\log_2|B|/2$, since its terminal nonempty leaves have size one. Reverse
the suffix tree for the reverse-order sweep.
\end{proof}

\begin{lemma}[One-sweep inclusion bound]\label{lem:inclusion}
Start from any deterministic $k$-set, and let $A$ be its occupied set after
one sweep. Then, for every $B\subseteq V$,
\begin{equation}\label{eq:inclusion}
 \Prb(B\subseteq A)\le(k/n)^{|B|}.
\end{equation}
This also holds for a reverse-order sweep.
\end{lemma}
\begin{proof}
The generating polynomial of a deterministic set is a monomial and is
stable: it is nonzero when all variables have positive imaginary parts.
A measure with stable multiaffine generating polynomial and nonnegative
coefficients is strongly Rayleigh. Averaging a measure with its image under
a transposition is partial symmetrization. By \cite[Theorem~4.20]{BorceaBrandenLiggett2009}, it
preserves this property, and \cite[Theorem~4.9]{BorceaBrandenLiggett2009} supplies negative
association. Apply these statements to the independent edge swaps in
succession.

One labelled card is exactly uniform after a sweep: every coordinate
becomes a fresh fair bit at its sole update. Thus every vertex has
occupation probability $k/n$. Negative association, applied repeatedly to
distinct occupation indicators, proves \eqref{eq:inclusion}. If $|B|>k$ the
left side is zero. The argument is independent of the coordinate order.
\end{proof}

\section{A sharper unrestricted collision moment}
\label{sec:moment}
The preceding section reduced marked-card dependence to the random variable
$e_G(A)$: the number of edges induced by the occupied marked set $A$ in a fixed
collision graph $G$.  We now bound an exponential moment of this count.  The
proof expands the product over edges, groups selected edges into connected
components, and overcounts each component by an embedded spanning tree.  This
is the step that turns cube geometry plus negative dependence into a usable
analytic bound.

The next result uses only the inclusion bound, not the full
strong-Rayleigh property.

\begin{theorem}[Connected edge-set moment]\label{thm:moment}
Let $G$ be a simple graph on $n$ vertices, of maximum degree at most $d$,
satisfying \eqref{eq:iso}. Let $A$ be a random set of exactly $k\ge2$ vertices
satisfying \eqref{eq:inclusion}. For $\lambda>0$, define
\begin{equation}\label{eq:a}
 a=4d\frac{k}{n}(1-2^{-\lambda})(2k)^{\lambda/2}.
\end{equation}
If $a<1$, then
\begin{equation}\label{eq:moment}
 \E\,2^{\lambda e_G(A)}\le
 \exp\!\left(\frac{ka}{1-a}\right).
\end{equation}
\end{theorem}
\begin{proof}
Put $t=2^\lambda-1>0$ and $\theta=k/n$. Expand over subsets of edges:
\begin{equation}\label{eq:edge-expansion}
 2^{\lambda e_G(A)}
 =\prod_{e\in E(G)}\bigl(1+t\ind_{\{e\subseteq A\}}\bigr)
 =\sum_{F\subseteq E(G)}t^{|F|}\ind_{\{V(F)\subseteq A\}},
\end{equation}
where $V(F)$ is the set of endpoints of the chosen edges. Every nonzero term
has $|V(F)|\le k$. For $B\subseteq V$ with $2\le |B|\le k$, define
\[
 w(B)=\sum_{\substack{F\subseteq E(G[B])\\(B,F)\text{ connected}}}t^{|F|}.
\]
Decomposing $F$ into its connected components, using the inclusion bound
on their disjoint union, and then allowing all collections, overlaps, and
repetitions in an exponential series yields
\begin{equation}\label{eq:component-sum}
 \E\,2^{\lambda e_G(A)}
 \le\exp\!\left(\sum_{2\le|B|\le k}\theta^{|B|}w(B)\right).
\end{equation}
More explicitly, before the last enlargement the sum is over unordered
pairwise disjoint collections of connected vertex sets. Its summand is
the product of their weights $\theta^{|B|}w(B)$. All terms are nonnegative,
so the enlargement is an upper bound. It asserts no independence of
components.

If $|B|=v$ and $(B,F)$ is connected, $F$ contains a spanning tree. Summing
over all possible spanning trees can only overcount, so
\begin{equation}\label{eq:tree-overcount}
 w(B)\le\sum_{T\text{ spanning tree of }G[B]}
 t^{v-1}(1+t)^{e_G(B)-(v-1)}.
\end{equation}
The total number of pairs $(B,T)$ with $|B|=v$ is at most
$n(4d)^{v-1}$. To see this, encode a rooted plane version of $T$ by its root,
one of at most $4^{v-1}$ rooted plane tree shapes, and a neighbor index,
with at most $d$ choices, for each tree edge. For example, choose the
smallest-labelled root and a fixed ordering of children to assign one
encoding to every pair. Such an encoding determines the embedded tree,
so it is a valid overcount of \emph{tree pairs}, not merely of vertex sets.

By \eqref{eq:iso}, $(1+t)^{e_G(B)}\le v^{\lambda v/2}$. Also, for $v\ge2$,
\[
 v^{v/(v-1)}\le2v\le2k,
 \qquad
 v^{\lambda v/2}\le(2k)^{\lambda(v-1)/2}.
\]
Here $v\le2^{v-1}$ proves the first inequality. Therefore the exponent in
\eqref{eq:component-sum} is at most
\begin{align*}
 \sum_{v=2}^k n(4d)^{v-1}\theta^v
 \left(\frac t{1+t}\right)^{v-1}v^{\lambda v/2}
 &\le k\sum_{v=2}^k
 \left[4d\theta(1-2^{-\lambda})(2k)^{\lambda/2}\right]^{v-1}\\
 &\le\frac{ka}{1-a}.
\end{align*}
This proves \eqref{eq:moment}.
\end{proof}

\begin{remark}[The gain and the independence condition]
The preceding proof retains one factor $1-2^{-\lambda}$ for every edge of
a spanning tree. The old vertex-component bound omitted these factors,
giving $a_{\rm old}=4d(k/n)(2k)^{\lambda/2}$. Thus the new theorem is stronger
and remains useful as $\lambda\downarrow0$.

In the shuffle application, $A$ is produced by the first sweep, and $G$ is
the collision graph of an independent second sweep. One may fix the entire
second network first. Its independence of $A$ preserves
\eqref{eq:inclusion}, and the theorem is uniform over every network. There
is no restriction to a low-collision event.
\end{remark}

\section{From collision moments to a pointwise heat kernel}
\label{sec:kernel}
The collision moment is an average bound.  The final representation argument
needs a pointwise bound on the marked-card transition kernel.  We obtain it by
viewing two sweeps as an operator $K=P_k^2$, proving matching row and column
$L^p$ estimates, and then interpolating from $L^1$ to $L^\infty$ in $m$
steps.  Reverse-coordinate sweeps supply the column estimate; no reversibility
of a complete sweep is assumed.

\begin{theorem}[Finite-parameter marked-card smoothing]\label{thm:smoothing}
Let $m\ge2$ be an integer and $2\le k\le n$. Put
\begin{equation}\label{eq:am}
 a_{d,k,m}=4d\frac{k}{n}
 \bigl(1-2^{-1/(m-1)}\bigr)(2k)^{1/[2(m-1)]}.
\end{equation}
If $a_{d,k,m}<1$, then, for all ordered distinct endpoints $x,y$,
\begin{equation}\label{eq:kernel}
 N_kP_k^{2m}(x,y)\le
 \exp\!\left(\frac{(m-1)k a_{d,k,m}}{1-a_{d,k,m}}\right).
\end{equation}
In particular the sufficient condition
\begin{equation}\label{eq:base-condition}
 8(\log2)d\frac{k}{n}(2k)^{1/[2(m-1)]}\le1
\end{equation}
implies $N_kP_k^{2m}(x,y)\le e^k$. Both conclusions hold for inverse noise.
For $k=1$, the normalized kernel is already one after a single sweep.
\end{theorem}
\begin{proof}
Set $p=m/(m-1)$, $p'=m$, $\lambda=p-1$, $N=N_k$ and $K=P_k^2$.
Jensen's inequality over the first-sweep transition law gives
\begin{align}
 \frac1N\sum_y(NK(x,y))^p
 &\le\sum_zP_k(x,z)\frac1N\sum_y(NP_k(z,y))^p\nonumber\\
 &=\sum_zP_k(x,z)\sum_yP_k(z,y)(NP_k(z,y))^{p-1}\nonumber\\
 &\le\E_x2^{(p-1)C_2}
 \le\exp\!\left(\frac{k a_{d,k,m}}{1-a_{d,k,m}}\right),
 \label{eq:row-bound}
\end{align}
where $C_2$ is the second sweep's collision count. The penultimate inequality
uses Lemma~\ref{lem:routing} and $Nn^{-k}\le1$. The last uses
Theorem~\ref{thm:moment}, with the independent future graph as just explained.

Every coordinate-layer operator is self-adjoint, so $P_k^*$ is the actual
reverse-order sweep. Repeating \eqref{eq:row-bound} for $(P_k^*)^2$ gives the
identical column $L^p(\pi_k)$ bound. With
\[
 M=\exp\!\left(\frac{k a_{d,k,m}}{p(1-a_{d,k,m})}\right),
\]
Minkowski and H\"older therefore give, respectively,
$\norm K_{L^1\to L^p}\le M$ and
$\norm K_{L^{p'}\to L^\infty}\le M$.
Finite-dimensional Riesz--Thorin interpolation gives
\[
 \norm K_{L^r\to L^s}\le M,
 \qquad \frac1s=\frac1r-\frac1m,
 \qquad 1\le r\le m.
\]
Iterate at $r_j=m/(m-j)$, $j=0,\ldots,m$, with $r_0=1$ and
$r_m=\infty$. All $m$ applications lie in the permitted range. Hence
\[
 \norm{K^m}_{L^1\to L^\infty}\le M^m
 =\exp\!\left(\frac{(m-1)k a_{d,k,m}}{1-a_{d,k,m}}\right).
\]
For a nonnegative matrix on uniform probability space,
$\norm{K^m}_{L^1\to L^\infty}=N\max_{x,y}K^m(x,y)$; a normalized point mass
attains the maximum. This proves \eqref{eq:kernel}.

For the constant-base assertion, write
$c_*=4Ld(k/n)(2k)^{1/[2(m-1)]}$. The elementary bound
$(m-1)(1-2^{-1/(m-1)})\le L$ gives
$a_{d,k,m}\le c_* /(m-1)$. If $c_*\le1/2$, then
\[
 \frac{(m-1)a_{d,k,m}}{1-a_{d,k,m}}
 \le\frac{c_*}{1-c_* /(m-1)}\le1.
\]
This is precisely \eqref{eq:base-condition}. Reverse sweeps prove the
inverse claim. No reversibility of the sweep kernel is assumed.
\end{proof}

\section{Endpoint geometry and the two entropy losses}
\label{sec:information}
This is the geometric heart of the entropy argument.  The idea is to reveal
the endpoints of a fresh sweep from the largest input index downward.  Even
after higher endpoints are known, the next endpoint is uniform on an explicit
candidate set.  Inside that set we construct a symmetric random partner.  When
we forget the detailed endpoint information, the partner's preimage is exactly
uniform among the earlier input positions.  This produces two complementary
ways to certify that the sweep erased information.

This section supplies the shuffle-specific information inequalities needed
for the improved entropy clock. They apply to an arbitrary input law, not
only to a law started from the identity.

Let $X$ be one independent sweep as a position permutation, let $\mu$ be a
random arrangement mapping labels to positions, and put $Y=X\circ\mu$.
The entropy loss is
\begin{equation}\label{eq:delta}
 \Delta=\Ent(\mu)-\Ent(Y)=I(Y;X),
\end{equation}
because the Shannon entropy of $Y$ given $X$ equals that of $\mu$.
Index positions by the integers $0,\ldots,n-1$, using their bit expansions.
Set
\[
 \mathcal H_j=\sigma(X(i):i>j),\qquad A_j=X(j),
\]
and let $s(j)$ be the number of one-bits of $j$.

\begin{lemma}[Candidates and a symmetric partner]\label{lem:candidates}
Given $\mathcal H_j$, $A_j$ is uniform on a measurable candidate set $C_j$
of size $2^{s(j)}$, containing no endpoint of a higher input. For $j\ge1$
there is a measurable symmetric stochastic matrix $Q_j$ on $C_j$, with
zero diagonal, such that a draw $B_j$ from $Q_j(A_j,\cdot)$ using
independent auxiliary randomness satisfies
\begin{equation}\label{eq:uniform-partner}
 K_j:=X^{-1}(B_j)<j,
 \qquad \Prb(K_j=k)=1/j\quad(0\le k<j).
\end{equation}
The displayed probabilities average the fresh sweep. They are not
conditional probabilities given $\mathcal H_j$.
\end{lemma}
\begin{proof}
Induct through a cube of size $2a$. Its first $d-1$ layers are independent
sweeps on the lower and upper halves, and its last layer matches equal
lower coordinates across the halves. Write an upper input as $j=a+i$.
The exposed higher upper endpoints reveal the corresponding smaller upper
endpoints and the final switch bits at their distinct used coordinates.
Each prescribed pattern of those bits has probability
$2^{-(a-1-i)}$, independently of the smaller routing. It therefore does not
bias that routing. The unexposed endpoint of $a+i$ is uniform on
$C_i^{\rm upper}\times\{0,1\}$.

For a lower input, all upper endpoints have been exposed. They reveal the
entire smaller upper routing and every final matching bit. The last layer
is now a known bijection from lower smaller coordinates to the remaining
endpoints. It transports the candidate set from the smaller lower sweep.
This proves the candidate assertion, including measurability, by induction
starting with a one-vertex cube.

For lower inputs transport the smaller partner kernel by that known
bijection. For $j=a+i$, from $(x,b)$ choose $(x,1-b)$ with probability
$(a-i)/(a+i)$. With the remaining probability $2i/(a+i)$, choose $y$
according to the smaller $Q_i(x,\cdot)$, draw a fair independent bit $c$,
and choose $(y,c)$. If $i=0$ only the first branch is used. The resulting
kernel is symmetric, stochastic, and loop-free.

In the first branch the input at the other endpoint is uniform among the
$a$ lower inputs, by independence and one-card uniformity of the smaller
lower sweep. In the second branch the fair bit chooses an upper or lower
input with probability $1/2$. The upper input is uniform among the $i$
earlier upper inputs by induction; the lower input is uniform among the
$a$ lower inputs by its independent sweep. Consequently each lower input
has probability
\[
 \frac{a-i}{(a+i)a}+\frac{i}{(a+i)a}=\frac1{a+i},
\]
and each earlier upper input has probability
$(2i/(a+i))/(2i)=1/(a+i)$. All candidates exclude the higher inputs, and
the diagonal is excluded, so $K_j<j$ holds pointwise. This establishes
\eqref{eq:uniform-partner} and completes the induction.
\end{proof}

For each $j$, let
$\mathcal G_{j+1}=\sigma(\mu^{-1}(i):i>j)$. On an atom of this field, the
remaining label set has size $\ell=j+1$. Write $p_k$ for the conditional
law of $\mu^{-1}(k)$ on that atom, for every $k\le j$. All $p_k$ in a row
of the argument use this \emph{same} conditioning field; the matrix
$(p_k(a))$ on the remaining labels is doubly stochastic. Let $u_\ell$ be
uniform on those labels. The chain rules give
\begin{equation}\label{eq:chain}
 \Ent(\mu)=\sum_{j=1}^{n-1}\E D(p_j\Vert u_{j+1}),
 \qquad
 \Delta=\sum_{j=1}^{n-1}I(Y;A_j\mid\mathcal H_j).
\end{equation}
The omitted $j=0$ terms are zero.

Define, for a distribution $p$ on $\ell$ points,
\[
 V(p)=1-\frac1\ell\left(\sum_a\sqrt{p(a)}\right)^2,
 \qquad S=\sum_{j=1}^{n-1}\E V(p_j).
\]

\begin{lemma}[Two separate losses]\label{lem:losses}
There are deterministic probability weights $W_{jk}$, $k\le j$, with
$W_{00}=1$ and, for $j=a+i$ where $a$ is the largest power of two at most
$j$,
\begin{equation}\label{eq:W}
 W_{jk}=\frac1{2a}\ (k<a),
 \qquad W_{j,a+k}=\frac12W_{ik}\ (k\le i),
\end{equation}
such that, with $q_j=\sum_{k\le j}W_{jk}p_k$,
\begin{equation}\label{eq:two-losses}
 \Delta\ge L S,
 \qquad
 \Delta\ge\mathcal C:=\sum_{j=1}^{n-1}\E D(p_j\Vert q_j).
\end{equation}
The two lower bounds in \eqref{eq:two-losses} are not added.
\end{lemma}
\begin{proof}
Fix $j$, and abbreviate $\mathcal H_j,\mathcal G_{j+1}$ to
$\mathcal H,\mathcal G$. The field $\mathcal G$ is independent of
$(A_j,\mathcal H)$, and is recoverable from $(Y,\mathcal H)$ because
$\mu^{-1}(i)=Y^{-1}(X(i))$ for $i>j$. The conditional-information chain
rule thus gives
\[
 I(Y;A_j\mid\mathcal H)=I(Y;A_j\mid\mathcal H,\mathcal G).
\]
Adjoin the partner $B_j$ of Lemma~\ref{lem:candidates} and an independent
fair bit $\xi$. Orient $(A_j,B_j)$ as $(V_0,V_1)$ according to $\xi$.
Given $\mathcal H$, symmetry of $Q_j$ and uniformity of $A_j$ make the
pair exchangeable. Thus $\xi$ is independent of
$(\mathcal H,\mathcal G,V_0,V_1)$. The auxiliary draw conveys no extra
information about $Y$ given $A_j,\mathcal H,\mathcal G$. It follows that
\begin{align*}
 I(Y;A_j\mid\mathcal H,\mathcal G)
 &=I(Y;A_j,B_j,\xi\mid\mathcal H,\mathcal G)\\
 &=I(Y;V_0,V_1,\xi\mid\mathcal H,\mathcal G)\\
 &\ge I(Y;\xi\mid\mathcal H,\mathcal G,V_0,V_1)\\
 &=I(Y,\mathcal H,V_0,V_1;\xi\mid\mathcal G)\\
 &\ge I((Y^{-1}(V_0),Y^{-1}(V_1));\xi\mid\mathcal G).
\end{align*}
In the last line the endpoints and their exposing field have been
forgotten. Now \eqref{eq:uniform-partner}, independence of the sweep from
$\mu$, and the fact that $K_j$ ranges uniformly over positions other than
$j$ show that, on an atom of $\mathcal G$, the two oriented label laws are
$R_p$ and $R_p^T$, where $p=p_j$ and
\[
 R_p(a,b)=\frac{p(a)}{\ell-1}\quad(a\ne b),\qquad R_p(a,a)=0.
\]
Indeed, conditional on any initial arrangement with $\mu^{-1}(j)=a$,
a uniform other position carries a uniform one of the remaining labels.
Writing $\bar R_p=(R_p+R_p^T)/2$, the last information is
\[
 \Phi_\ell(p)=\tfrac12D(R_p\Vert\bar R_p)
              +\tfrac12D(R_p^T\Vert\bar R_p).
\]
The scalar inequality
\begin{equation}\label{eq:two-point}
 a\log\frac{2a}{a+b}+b\log\frac{2b}{a+b}
 \ge L(\sqrt a-\sqrt b)^2
\end{equation}
implies
$\Phi_\ell(p)\ge \ell L V(p)/(\ell-1)\ge LV(p)$ after summing pairs.
For completeness, by homogeneity set $b=1$. If $F$ is half the left side
of \eqref{eq:two-point} and $h(x)=(\sqrt x-1)^2$, then
$F''(x)/h''(x)=\sqrt x/(1+x)$ decreases for $x>1$. Both functions and their
first derivatives vanish at one. Successive weighted-average integrations
therefore show that $F/h$ decreases to $L/2$. Symmetry under $x\mapsto1/x$
and continuity prove the remaining cases. Summation in \eqref{eq:chain}
gives the first loss.

For the second loss, compare the conditional joint law of $(A_j,Y)$ with
the product of its marginals and apply the map $(a,y)\mapsto y^{-1}(a)$.
Under the joint law the image has law $p_j$. Under the product law it has
law $\sum_{k\le j}w_{jk}(\mathcal H)p_k$, where
\[
 w_{jk}(\mathcal H)=
 \frac{\Prb(X(k)\in C_j\mid\mathcal H)}{|C_j|}.
\]
Data processing gives the corresponding relative-entropy lower bound.
Convexity in the reference distribution, and independence of $\mathcal G$
from the sweep, replace these random weights by
$W_{jk}=\E w_{jk}$. They are probability weights and exclude higher inputs.

The half-cube decomposition verifies \eqref{eq:W}. For an upper input
$j=a+i$, a lower input falls in the smaller upper candidate coordinates
with probability $|C_i|/a$, by the independent uniform one-card lower
sweep. Dividing by $|C_j|=2|C_i|$ gives $1/(2a)$. Upper inputs inherit half
the smaller weights. Lower inputs inherit the smaller weights through the
known final bijection. Sum over \eqref{eq:chain} to obtain the second loss.
\end{proof}

\section{A dimension-free scalar interpolation inequality}
\label{sec:scalar}
The entropy proof now becomes a one-dimensional comparison between three
quantities for a probability vector: relative entropy from uniformity,
relative entropy from a reference law, and square-root overlap.  The next
lemma is deliberately dimension-free so that its constants do not deteriorate
when the number of still-unrevealed labels is large.

\begin{lemma}[Elementary scalar truncation]\label{lem:scalar-trunc}
For $T\ge4$ and $z\ge0$,
\begin{equation}\label{eq:scalar-trunc}
 z\log z-z+1\le(\log T+2)(\sqrt z-1)^2+z\log^+(z/T),
\end{equation}
with $0\log0=0$.
\end{lemma}
\begin{proof}
For $z\le1$, differentiating in $\sqrt z$ gives
$z\log z-z+1\le2(\sqrt z-1)^2$.
For $1\le z\le T$, put $y=\sqrt z$. The difference between
$(\log z+2)(\sqrt z-1)^2$ and $z\log z-z+1$ is
\[
 3y^2-4y+1+2(1-2y)\log y.
\]
Its value and first derivative vanish at one, and its second derivative is
$2(3y+1)(y-1)/y^2\ge0$. Increasing the coefficient from $\log z+2$ to
$\log T+2$ preserves the inequality. For $z\ge T$, the right side minus
the left side in \eqref{eq:scalar-trunc} is
\[
 3y^2-2(\log T+2)y+\log T+1.
\]
It is nonnegative at $y=\sqrt T$ by the previous case and is increasing
thereafter: $3\sqrt T\ge\log T+2$ for $T\ge4$.
\end{proof}

\begin{lemma}[Uniform scalar estimate]\label{lem:scalar}
Let $p,q$ be distributions on $\ell$ points. Suppose
\[
 \ell^{1/2}\sum_a q(a)^{3/2}\le R,
 \qquad 1\le R\le\ell^{1/2}.
\]
For $V=1-\ell^{-1}(\sum_a\sqrt{p(a)})^2$,
\begin{equation}\label{eq:scalar}
 D(p\Vert u_\ell)
 \le3D(p\Vert q)+\bigl[4\log(R/V)+8\bigr]V.
\end{equation}
The final term is defined as zero when $V=0$.
\end{lemma}
\begin{proof}
Write $f=\ell p$, $A=\E_{u_\ell}\sqrt f=\sqrt{1-V}$, and
$h(z)=z\log z-z+1$. Then
\[
 D(p\Vert u_\ell)\le\E_{u_\ell}A^2h(f/A^2),
\]
because the right side minus the left is $-\log A^2-1+A^2\ge0$.
If $V=0$, $p$ is uniform and \eqref{eq:scalar} is immediate.

First suppose $0<V\le1/2$, so $A^2\ge1/2$. Set
$M=8(R/V)^2$ and apply Lemma~\ref{lem:scalar-trunc} with $T=M/A^2$.
The centered square integrates to $V$, and hence
\begin{equation}\label{eq:smallV}
 D(p\Vert u_\ell)\le(\log M-\log A^2+2)V+T_1,
 \qquad
 T_1=\sum_{f(a)>M}p(a)\log(f(a)/M).
\end{equation}
The entropy variational inequality applied to
$g=\frac13\log^+(f/M)$ gives
\[
 D(p\Vert q)\ge\tfrac13T_1-\log\E_qe^g.
\]
This variational inequality follows directly by comparing $p$ with the
probability distribution proportional to $e^gq$ and using nonnegativity
of relative entropy. By H\"older,
\[
 \E_q e^g\le1+M^{-1/3}R^{2/3}p\{f>M\}^{1/3}.
\]
When $f>M\ge8$, one has $f\le4(\sqrt f-1)^2$. Thus
\[
 p\{f>M\}\le8(1-A)\le8V.
\]
The extra term in the bound on $\E_qe^g$ is at most $V$ by the choice of
$M$. Consequently $T_1\le3D(p\Vert q)+3V$. Substituting into
\eqref{eq:smallV}, and using $-\log A^2\le\log2$ and
$\log16+5<8$, proves the stronger bound with $2\log(R/V)$ in place of
$4\log(R/V)$.

If $V>1/2$, apply the entropy variational inequality to
$\frac12\log(\ell q)$. Combining it with
$D(p\Vert u_\ell)=D(p\Vert q)+\E_p\log(\ell q)$ yields
\[
 D(p\Vert u_\ell)\le3D(p\Vert q)+2\log R
 \le3D(p\Vert q)+4V\log(R/V).
\]
This proves \eqref{eq:scalar}. Zeros follow by limiting conventions; if
$D(p\Vert q)=\infty$ the conclusion is automatic.
\end{proof}

\section{The improved entropy profile and clock}
\label{sec:entropy}
We can now close the feedback loop.  The partner weights have a bounded summed
moment below a critical exponent.  Combining that moment with the two separate
entropy losses yields a differential-style statement: when the current entropy
is $E$, one sweep removes at least a controlled fraction of $E$.  Integrating
that profile over successive entropy scales gives an explicit number of sweeps
needed to reach any target $B$.

\begin{lemma}[A subcritical endpoint moment]\label{lem:W-moment}
For the weights of \eqref{eq:W}, define
\[
 R_j=(j+1)^{1/2}\sum_{k\le j}W_{jk}^{3/2},
 \qquad C_R=3+2\sqrt2<6.
\]
On every conditioning atom in Section~\ref{sec:information},
\begin{equation}\label{eq:qmoment}
 (j+1)^{1/2}\sum_a q_j(a)^{3/2}\le R_j,
 \qquad 1\le R_j\le(j+1)^{1/2},
 \qquad \sum_{j=1}^{n-1}R_j\le C_Rn.
\end{equation}
\end{lemma}
\begin{proof}
Doubly stochasticity gives, for each label $a$,
\[
 \left(\sum_{k\le j}p_k(a)W_{jk}\right)^{3/2}
 \le\sum_{k\le j}p_k(a)W_{jk}^{3/2}.
\]
Sum over $a$ to obtain the first assertion. The other bounds for an
individual $R_j$ follow from convexity and the fact that $W_j$ is a
probability row.

Let $U_j=\sum_{k\le j}W_{jk}^{3/2}$ and $T_d=\sum_{j<2^d}U_j$.
The recursion \eqref{eq:W} gives exactly
\begin{equation}\label{eq:recurrence}
 T_d=(1+2^{-3/2})T_{d-1}+2^{-3/2}2^{(d-1)/2},
 \qquad T_0=1.
\end{equation}
Dividing by $2^{d/2}$ yields a scalar affine recursion with multiplier
$1/\sqrt2+1/4<1$ and additive term $1/4$. Its fixed point is
$3+2\sqrt2=C_R$, which exceeds the initial value. Hence
$T_d\le C_R2^{d/2}$, and
$\sum_{j=1}^{n-1}R_j\le n^{1/2}T_d\le C_Rn$.
\end{proof}

\begin{theorem}[Dimension-free entropy profile]\label{thm:entropy}
For every nonuniform initial law $\mu$, one independent sweep has
$\Delta=\Ent(\mu)-\Ent(X\circ\mu)>0$ and
\begin{equation}\label{eq:entropy-profile}
 \frac{\Ent(\mu)}{\Delta}
 \le100+12\log^+\!\left(\frac n{\Ent(\mu)}\right).
\end{equation}
\end{theorem}
\begin{proof}
Apply Lemma~\ref{lem:scalar} on each conditioning atom and sum using
\eqref{eq:chain}, \eqref{eq:qmoment}, and the log-sum inequality. Write
$E=\Ent(\mu)$ and $R_\Sigma=\sum_{j=1}^{n-1}R_j$. Then
\begin{equation}\label{eq:ECS}
 E\le3\mathcal C+4S\log(R_\Sigma/S)+8S.
\end{equation}
If $S=0$, each $p_j$ is uniform and $E=0$. Thus $E>0$ implies $S>0$ and,
by \eqref{eq:two-losses}, $\Delta>0$.

Put $c=1/L\le3/2$ and $x=E/\Delta\ge1$. The two separate inequalities
in \eqref{eq:two-losses} give $\mathcal C\le\Delta$ and $S\le c\Delta$.
For $v=S/(c\Delta)\in(0,1]$, the elementary bound
$v\log(1/v)\le1/e$ yields
\begin{align*}
 S\log(C_Rn/S)
 &=c\Delta\left[v\log\frac{C_Rn}{c\Delta}+v\log(1/v)\right]\\
 &\le c\Delta\left[\log^+\frac{C_Rn}{c\Delta}+1/e\right].
\end{align*}
This avoids any monotonicity assumption on $s\log(C_Rn/s)$ outside its
increasing range. Dividing \eqref{eq:ECS} by $\Delta$ gives
\[
 x\le3+8c+4c/e+4c\log^+\frac{C_Rn}{c\Delta}.
\]
Since $C_R/c<6$, $\log6<2$, and $x\ge1$, it follows that
\[
 x\le30+6\log^+(n/E)+6\log x.
\]
Concavity of the logarithm gives
$\log x\le x/12+\log12-1\le x/12+2$. Absorb $x/2$ on the right to obtain
$x\le84+12\log^+(n/E)$, and enlarge $84$ to $100$.
\end{proof}

\begin{corollary}[Explicit sweep clock]\label{cor:clock}
For $1\le B\le n$, define
\begin{align}
 J(B)&=\left\lceil\log_2(n/B)\right\rceil,\nonumber\\
 R(B)&=\left\lceil100\log^+(dL)\right\rceil+
 \sum_{j=1}^{J(B)}\left\lceil L(100+12jL)\right\rceil.
 \label{eq:clock}
\end{align}
Then $\Ent(q_{R(B)})\le B$ and, with an absolute implicit constant,
\begin{equation}\label{eq:clock-order}
 R(B)=O\!\left(\log(d+2)+J(B)^2\right).
\end{equation}
\end{corollary}
\begin{proof}
Initially $\Ent(q_0)=\log(n!)\le ndL$. While entropy exceeds $n$,
Theorem~\ref{thm:entropy} contracts it by a factor at most
$1-1/100\le e^{-1/100}$ per sweep. The first term in \eqref{eq:clock}
therefore brings it to at most $n$. While it is between $n2^{-j}$ and
$n2^{-(j-1)}$, the denominator in \eqref{eq:entropy-profile} is at most
$100+12jL$. The $j$th summand suffices to halve the previous threshold,
unless that threshold has already been crossed. The last dyadic threshold
is at most $B$. Summing the affine-in-$j$ ceilings proves
\eqref{eq:clock-order}.
\end{proof}

\begin{remark}[What changed in the entropy proof]
The exponent $3/2$ is subcritical for the endpoint-weight recursion and
therefore gives a summed moment of order $n$.  The scalar estimate also avoids
a dimension-dependent split near zero square-root overlap.  These features are
what allow the entropy denominator used in the main argument to be
logarithmic with constants independent of $d$.
\end{remark}

\section*{Acknowledgments}
Part of this research was carried out at AxiomMath AI.  The author thanks
Evan Chen, Vasily Ilin, Simon Mahns, Ken Ono, and Jesse Thorner for fruitful
discussions around AI for mathematics.  The author also acknowledges
substantial assistance from OpenAI's ChatGPT in exploratory calculations, proof
checking, and preparation of the manuscript.  The mathematical claims and any
remaining errors are the author's responsibility.

\end{document}